\documentclass[10pt,reqno]{amsart}
\usepackage{amsmath,amscd,amssymb}
\usepackage{url}
\usepackage{graphicx}
\usepackage{color}
\usepackage{subfigure}

\theoremstyle{plain}
   \newtheorem{theorem}{Theorem}[section]
   
   \newtheorem{lemma}[theorem]{Lemma}
   \newtheorem{corollary}[theorem]{Corollary}
   \newtheorem{conjecture}[theorem]{Conjecture}
   
\theoremstyle{definition}

\theoremstyle{remark}
 \newtheorem{remark}{Remark}[section]

\newcommand{\R}{\mathbb{R}}
\newcommand{\RR}{\mathcal{R}}
\newcommand{\Z}{\mathbb{Z}}

\newcommand{\B}{\mathcal{B}}
\newcommand{\cI}{\mathcal{I}}

\newcommand{\zone}{\mathcal{Z}}

\newcommand{\bm}[1]{{\boldsymbol{#1}}}
\def\0{\bm{0}}
\def\1{\bm{1}}

\def\be{\bm{e}}
\def\m{\bm{m}}
\def\bp{\bm{p}}
\def\bq{\bm{q}}
\def\bs{\bm{s}}
\def\bu{\bm{u}}
\def\bv{\bm{v}}
\def\bx{\bm{x}}

\def\bomega{\bm{\omega}}

\def\newop#1{\expandafter\def\csname #1\endcsname{\mathop{\rm
#1}\nolimits}}

\newop{Ehr}
\newop{vol}
\newop{Vol}
\newop{conv}
\newop{Supp}
\newop{supp}
\newop{aw-height}
\newop{Nasc}
\newop{nasc}
\newop{comp}
\newop{Des}
\newop{des}
\newop{maxDes}
\newop{awheight}
\newop{lcm}
\newop{red}
\newop{Int}
\newop{ch}
\newop{Vert}
\newop{Bey}
\newop{asc}
\newop{Pyr}

\subjclass[2000]{}

\begin{document}
\title[Ehrhart Unimodality and Ehrhart Positivity]{On the Relationship Between Ehrhart Unimodality and Ehrhart Positivity}

\author{Fu Liu}
\date{\today}
\address{Department of Mathematics, University of California, Davis, One Shields Avenue, Davis, CA, 95616, USA}
\email{fuliu@math.ucdavis.edu}

\author{Liam Solus}
\date{\today}
\address{Matematik, KTH, SE-100 44 Stockholm, Sweden}
\email{solus@kth.se}

\begin{abstract}
For a given lattice polytope, two fundamental problems within the field of Ehrhart theory are to (1) determine if its (Ehrhart) $h^\ast$-polynomial is unimodal and (2) to determine if its Ehrhart polynomial has only positive coefficients.  
The former property of a lattice polytope is known as Ehrhart unimodality and the latter property is known as Ehrhart positivity.  
These two properties are often simultaneously conjectured to hold for interesting families of lattice polytopes, yet they are typically studied in parallel.  
As to answer a question posed at the 2017 Introductory Workshop to the MSRI Semester on Geometric and Topological Combinatorics, the purpose of this note is to show that there is no general implication between these two properties in any dimension greater than two.  
To do so, we investigate these two properties for families of well-studied lattice polytopes, assessing one property where previously only the other had been considered.  
Consequently, new examples of each phenomena are developed, some of which provide an answer to an open problem in the literature. 
The well-studied families of lattice polytopes considered include zonotopes, matroid polytopes, simplices of weighted projective spaces, empty lattice simplices, smooth polytopes, and $s$-lecture hall simplices.
\end{abstract}

\maketitle
\thispagestyle{empty}

\section{Introduction}
\label{sec: introduction}
A subset $P\subset\R^n$ of $n$-dimensional real Euclidean space is called a \emph{(lattice) polytope} if it is the convex hull of finitely many \emph{lattice points} (i.e. points in $\Z^n$) that together span a $d$-dimensional affine subspace of $\R^n$.  
Lattice polytopes play a central role in geometric and algebraic combinatorics and algebraic geometry.  
In the former context, lattice polytopes are often associated to combinatorial and algebraic objects such that their geometry reflects known facts, and/or reveals new and interesting facts, about these objects.  
In the latter context, each lattice polytope serves as a ``polyhedral dictionary'' from which we can read the algebro-geometric properties of an associated \emph{toric variety}.  
Consequently, lattice polytopes amount to a large and diverse family of examples within algebraic geometry.  
In both fields of research, the number of lattice points within the \emph{$t^{th}$ dilate} of a lattice polytope $P$, $tP:=\{t\bp\in\R^n : \bp\in P\}$, provides information about the associated algebraic and geometric objects.  
The \emph{Ehrhart function} of a $d$-dimensional lattice polytope $P$ is the function $i(P;t):=|tP\cap\Z^n|$ for $t\in\Z_{\geq0}$.  
It is well-known \cite{E62} that $i(P;t)$ is a polynomial in $t$ of degree $d$ (called the \emph{Ehrhart polynomial of $P$}), and the \emph{Ehrhart series} of $P$ is the rational function
$$
\Ehr_p(z) := \sum_{t\geq0}i(P;t)z^t = \frac{h_0^*+h_1^*z+\cdots+h_d^*z^d}{(1-z)^{d+1}},
$$
where the coefficients $h^*_0,h^*_1,\ldots,h^*_d$ are all nonnegative integers \cite{S80}.  
The polynomial $h^*(P;z):=h_0^*+h_1^*z+\cdots+h_d^*z^d$ is called the \emph{$h^*$-polynomial} of $P$. 
In the field of \emph{Ehrhart theory}, the lattice point combinatorics of a polytope $P$ are studied using both its Ehrhart polynomial and its $h^\ast$-polynomial.  
Both $i(P;t)$ and $h^\ast(P;z)$ often reflect interesting properties of the underlying combinatorial structure of $P$, but the questions and techniques common to the study of each polynomial are drastically different.  
Consequently, although the Ehrhart polynomial and the $h^\ast$-polynomial may be analyzed simultaneously, they are often studied in parallel.  
In particular, two popular, and parallel, endeavors in Ehrhart theory are to analyze when a lattice polytope $P$ is \emph{Ehrhart positive} and when it is \emph{Ehrhart unimodal}.  

A lattice polytope $P$ is called \emph{Ehrhart positive} if all the coefficients of the Ehrhart polynomial $i(P;t)$ are positive rational numbers, and it is called \emph{Ehrhart unimodal} if its $h^\ast$-polynomial is \emph{unimodal}; i.e. for some $t\in[d]:=\{1, 2, \dots, d\}$ we have that 
$$
h_0^\ast\leq h_1^\ast\leq\cdots\leq h_t^\ast\geq\cdots\geq h_{d-1}^\ast\geq h_d^\ast.
$$
Both Ehrhart positivity and Ehrhart unimodality are popularly investigated properties with deep algebraic and geometric underpinnings.  
In fact, each property was recently the focus of its own survey article \cite{B16,L17}, and it is common to see both properties conjectured to hold for nice lattice polytopes (see for example \cite[Conjecture~2]{DHK09}).  
At the \emph{2017 MSRI Introductory Workshop to the Semester on Geometric and Topological Combinatorics}, the first author gave a talk on problems and progress in Ehrhart positivity and the second author gave an analogous talk on Ehrhart unimodality.  
Following these talks, A. Postnikov posed the question as to whether or not there exists any implication between the two problems.  
The purpose of the note is to answer this question in each dimension.  
In doing so, we also address this question for some major families of well-studied lattice polytopes.  

In the remainder of this note, we show by way of examples in each dimension greater than two that there is no general implication between Ehrhart positivity and Ehrhart unimodality.  
In doing so, we study the relationship between Ehrhart positivity and Ehrhart unimodality by way of the fundamental examples associated to each property.  
We determine whether or not classic examples of one property satisfy the other.  
In section~\ref{sec: polytopes that are ehrhart positive and ehrhart unimodal} we summarize the lattice polytopes that are known (or conjectured) to be both Ehrhart positive and Ehrhart unimodal.  
In section~\ref{sec: polytopes that are ehrhart positive but not ehrhart unimodal} we present lattice polytopes that are Ehrhart positive but not Ehrhart unimodal.  
In doing so, we provide an answer to an open problem posed in \cite{BR07}.
In section~\ref{sec: polytopes that are not ehrhart positive but are ehrhart unimodal} we present examples that are not Ehrhart positive but are Ehrhart unimodal.  
Finally, in section~\ref{sec: polytopes that are neither ehrhart positive nor ehrhart unimodal} we describe families of lattice polytopes that are neither Ehrhart positive nor Ehrhart unimodal.  
The examples considered here are all classic families of polytopes, including zonotopes, matroid polytopes, simplices of weighted projective spaces, empty lattice simplices, smooth polytopes, and $s$-lecture hall simplices.  

\subsection{Preliminaries}
\label{subsec: preliminaries}
Before we begin, we briefly catalogue some basic facts about Ehrhart theory that will be used in the remainder of this note.
Let $P\subset \R^n$ be a $d$-dimensional lattice polytope.  
The first important fact we need is that the Ehrhart polynomial of $P$ can be recovered from its $h^\ast$-polynomial $h^*(P;z) = \sum_{j=0}^d h^*_j z^j$ by way of the formula
\begin{equation}
i(P,t) = \sum_{j=0}^d h^*_j \binom{t + d -j}{d}.  \label{equ:h2e}
\end{equation}
Next, recall from the introduction that the coefficients of $h^\ast(P;z)$ are known to be nonnegative integers \cite{S80}.  
Going a step beyond this, some of the coefficients of $h^\ast(P;z)$ have very simply-stated formulae.
In particular, we know that
\[
h_0^\ast = 1,
\qquad
\mbox{and}
\qquad
h_1^\ast = |P\cap \Z^n| - (d+1).
\]
Finally, in the following we will utilize some nice implications that hold between properties of the roots of a univariate polynomial and the distribution of its coefficeints.  
A univariate polynomial is called \emph{real-rooted} if all of its roots are real numbers.  
It turns out that if this polynomial further has nonnegative coefficients then it is unimodal \cite[Theorem 1.2.1]{B89}.  
In the following, we will often use the fact that a given $h^\ast$-polynomial has only real-roots to recover Ehrhart unimodality.  
%
%

\section{Polytopes that are Ehrhart Positive and Ehrhart Unimodal}
\label{sec: polytopes that are ehrhart positive and ehrhart unimodal}
The major conjectures on Ehrhart positivity and Ehrhart unimodality are naturally aimed at positive results; i.e., they purport that a given family of lattice polytopes satisfies the desired property.  
Consequently, to identify families of lattice polytopes in each dimension that are both Ehrhart positive and Ehrhart unimodal, it suffices to compare the positive results in both fields and identify where they overlap.  
Moreover, substantially difficult conjectures on Ehrhart positivity are often stated in parallel to equally challenging conjectures on Ehrhart unimodality.  
In subsection~\ref{subsec: known families} we catalogue the families of lattice polytopes that are known to be both Ehrhart positive and Ehrhart unimodal. 
Then, in subsection~\ref{subsec: conjectured families}, we review which families of lattice polytopes are further conjectured to satisfy both properties.  
However, we first begin by assessing our question in dimension two; i.e.~the case of all \emph{lattice polygons}.  

\subsection{Dimension two:~the lattice polygons}
\label{subsec: dimension two}
Since the goal of this note is to assess the relationship (or lack thereof) between Ehrhart unimodality and Ehrhart positivity in each dimension, then it is natural to first consider our question in dimension two.  
In fact, dimension two turns out to be the only dimension in which there is a definitive relationship between these two properties!
A two-dimensional lattice polytope is often called a \emph{lattice polygon}.  
It follows from \emph{Pick's Theorem} \cite[Theorem 2.8]{BR07} that if $P\subset\R^2$ is a lattice polygon then
\[
i(P;t) = At^2+\frac{1}{2}Bt+1
\quad
\mbox{and}
\quad
h^\ast(P;z) = \left(A-\frac{B}{2}+1\right)z^2+\left(A+\frac{B}{2}-2\right)z+1,  
\]
where $A$ denotes the area of $P$ and $B$ denotes the number of lattice points on the boundary of $P$.  
It can be seen directly from these formulae that $P$ is both Ehrhart positive and Ehrhart unimodal.  
In particular, Ehrhart unimodality follows from the observation that $A\geq 1/2$ and $B\geq 3$. 
\begin{remark}
\label{rmk: higher dimensions}
Since lattice polygons are both Ehrhart positive and Ehrhart unimodal, in what remains we only consider examples in dimensions greater than two.
\end{remark}

\subsection{Known families}
\label{subsec: known families}
For dimensions greater than two, perhaps the simplest example of a lattice polytope that is both Ehrhart positive and Ehrhart unimodal is the \emph{standard $d$-simplex}, which is the convex hull 
$$
\Delta_d := \conv(\be_1,\ldots,\be_d,\0)\subset\R^d,
$$
where $\be_1,\ldots, \be_d$ denote the standard basis vectors and $\0$ denotes the origin in $\R^d$.  
It is well-known \cite[Chapter 2.3]{BR07} that 
\[
i(\Delta_d;t) = {t+d\choose d} 
\qquad
\mbox{and}
\qquad
h^\ast(\Delta_d;z) = 1,
\]
from which it is straight-forward to see that $\Delta_d$ is both Ehrhart positive and Ehrhart unimodal.  
A second famous example is the \emph{$d$-dimensional cross-polytope}, which is defined as the convex hull
$$
\diamondsuit_d :=\conv(\be_1,\ldots,\be_d,-\be_1,\ldots,-\be_d)\subset\R^d.
$$
The Ehrhart polynomial and the $h^\ast$-polynomial of $\diamondsuit_d$ are, respectively \cite[Chapter 2.5]{BR07},
$$
i(\diamondsuit_d;t) = \sum_{k=0}^d{d\choose k}{t-k+d\choose d}
\qquad
\mbox{and}
\qquad
h^\ast(\diamondsuit_d;z) = (z+1)^d.  
$$
Since $h^\ast(P;z)$ is seen to be real-rooted, the unimodality of its coefficients follows from the discussion in subsection~\ref{subsec: preliminaries}.  
At the same time, \cite[Exercise 4.61(b)]{S97} demonstrates that every zero of $i(\diamondsuit_d;t)$ has real part $-\frac{1}{2}$.  
Therefore, $i(\diamondsuit_d;t)$ is a product of polynomials of the form
$$
t+\frac{1}{2}
\qquad
\mbox{and}
\qquad
t^2+t+\frac{1}{4}+a^2,
$$
for some $a\in\R$,and so $\diamondsuit_d$ is also Ehrhart positive.

Finally, perhaps the most substantial family of lattice polytopes that are known to be both Ehrhart positive and Ehrhart unimodal are the lattice \emph{zonotopes}.  
A lattice polytope $\zone$ is called a \emph{zonotope} if it is the \emph{Minkowski sum} of a collection of line segments; i.e., $\zone$ is translation-equivalent to 
$$
\{\lambda_1 \bu_1 + \lambda_2 \bu_2 +\cdots+\lambda_m \bu_m \ : \ 1 \le \lambda_k \le 0, \  \forall 1 \le k \le m\},
$$
for some $\bu_1,\ldots,\bu_m \in\Z^n$.  
Zonotopes include a wide variety of lattice polytopes such as the $d$-dimensional unit cube $[0,1]^d$ and the \emph{regular permutahedron} \cite[Example 0.10]{Z12}.  
The former example is known to have Ehrhart polynomial $(t+1)^d$ and $h^\ast$-polynomial the \emph{$n^{th}$ Eulerian polynomial} $A_n(z)$ \cite[Chapter 2.2]{BR07}, and the latter example's Ehrhart polynomial has $i^{th}$ coefficient being the number of forests on $d+1$ vertices with $i$ edges \cite[Example 3.1]{S80}. 
More generally, zonotopes are seen to be Ehrhart positive by way of a combinatorial formula for the coefficients of $i(\zone;t)$ in terms of the vectors $\bu_1,\ldots,\bu_m$ \cite[Example 3.1]{S80}, 
It is also known that the $h^\ast$-polynomial of any zonotope is real-rooted (and unimodal) \cite[Theorem 1.2]{BJM16}.

\subsection{Conjectured families}
\label{subsec: conjectured families}
While zonotopes constitute the major family of lattice polytopes known to be both Ehrhart positive and Ehrhart unimodal, there do exist other large families of lattice polytopes that are conjectured to satisfy both conditions.  
One substantial family of such polytopes are the \emph{matroid polytopes}.  
Recall that a \emph{matroid} $M$ is a finite collection $\cI$ of subsets of $[d] := \{1,\ldots,d\}$ that satisfies the following three properties:
\begin{enumerate}
	\item $\emptyset\in \cI$,
	\item if $A\in \cI$ and $B\subseteq A$ then $B\in \cI$, and 
	\item if $A,B\in \cI$ and $|A| = |B|+1$ then there exists $i \in A\backslash B$ so that $B\cup\{i\}\in \cI$.  
\end{enumerate} 
The elements of $\cI$ are called the \emph{independent sets} of $M$ and the inclusion-maximal independent sets are called its \emph{bases}.  
If $\B$ denotes the collection of bases of a matroid $M$, then we defined the matroid polytope for $M$ to be the convex hull
$$
P(M) := \conv\left(\sum_{i\in B}\be_i : B\in \B\right)\subset \R^d.
$$
\begin{conjecture}
\cite[Conjecture 2]{DHK09}
\label{conj: deloera, haws, and koeppe}
For any matroid $M$, the matroid polytope $P(M)$ is both Ehrhart positive and Ehrhart unimodal. 
\end{conjecture}
So far, both aspects of this conjecture have remained elusive despite various attempts and partial results \cite{BValpha,ehrhartpos-gp-fpsac,KMR17}. 
In general, families for which it is easy to prove Ehrhart unimodality may not be amenable to proofs of Ehrhart positivity (or vice versa).  
For instance, a family of lattice simplices known as the \emph{simplices for base-$r$ numeral systems}, whose combinatorics are tied to representations of integers in the base-$r$ numeral system, were recently shown to have real-rooted (and therefore unimodal) $h^\ast$-polynomials \cite[Theorem 4.5]{S17}.  
Given an integer $r\geq 2$, the \emph{base-$r$ $d$-simplex} is defined to be 
$$
\B_{(r,d)} := \conv\left(\be_1,\ldots,\be_d,-\sum_{k=1}^{d}(r-1) r^{k-1} \be_{k}\right)\subset\R^d.
$$
Based on observed data, the author of \cite{S17}, further conjectured that such simplices also satisfy Ehrhart positivity.
\begin{conjecture}
\cite[Section 5]{S17}
For $r\geq 2$ and $d\geq 1$, the base-$r$ $d$-simplex is Ehrhart positive.  
\end{conjecture}
\section{Polytopes that are Ehrhart Positive but not Ehrhart Unimodal}
\label{sec: polytopes that are ehrhart positive but not ehrhart unimodal}
In this section, we present lattice polytopes in each dimension greater than two that are Ehrhart positive but not Ehrhart unimodal.  
To start, we define for every weakly increasing vector of positive integers $\bq = (q_1,\ldots,q_d)\in\R^d$ a lattice $d$-simplex 
\[
\Delta_{(1,\bq)} := \conv(\be_1,\ldots,\be_d, -\bq)\subset\R^d.
\]
These lattice simplices have been studied extensively from the perspective of Ehrhart unimodality \cite{BD16,BDS16,P08,S17}.  
For instance, when $\bq = ((r-1),(r-1)r,\ldots,(r-1)r^{d-1})$ for some $r\geq 2$, then $\Delta_{(1,\bq)}$ is the base-$r$ $d$-simplex $\B_{(r,d)}$ described in section~\ref{sec: polytopes that are ehrhart positive and ehrhart unimodal}.  
Moreover, in \cite{P08} it is shown that for special choices of $\bq$, the $h^\ast$-polynomial of $\Delta_{(1,\bq)}$ is non-unimodal.  
These examples refuted (in all dimensions greater than $5$) the conjecture of Hibi \cite{H92} that every Gorenstein lattice polytope has a unimodal $h^\ast$-polynomial.
\begin{theorem}
\cite{P08}
\label{thm: paynes theorem}
Let $r\geq0$, $s\geq 3$, and $k\geq r+2$ be integers.  If 
$$
\bq = (q_1,\ldots,q_d)=(\underbrace{1,1,\ldots,1}_{sk-1 \text{ times}},\underbrace{s,s,\ldots,s}_{r+1 \text{ times}}),
$$
then 
$$
h^\ast(\Delta_{(1,\bq)};z) = (1+z^k+z^{2k}+\cdots +z^{(s-1)k})(1+z+z^2+\cdots +z^{k+r}).
$$
Therefore, $\Delta_{(1,\bq)}$ is not Ehrhart unimodal.
\end{theorem}
On the other hand, we can show that for every $\bq$ as in Theorem~\ref{thm: paynes theorem} the simplex $\Delta_{(1,\bq)}$ is Ehrhart positive.  
This constitutes a new class of Ehrhart positive lattice polytopes, and this is the first such class of polytopes that are known to be Ehrhart positive but not Ehrhart unimodal.  
\begin{theorem}
\label{thm: ehrhart positive but not ehrhart unimodal q}
For integers $r\geq0$, $s\geq 3$, and $k\geq r+2$, let $\Delta_{(1,\bq)}$ be defined as in Theorem~\ref{thm: paynes theorem}.  Then $\Delta_{(1,\bq)}$ is Ehrhart positive.  
\end{theorem}

\begin{proof}
	Notice that the zeros of $h^*(\Delta_{(1,\bq)}; z)$ are all on the unit circle $\{ z \in {\mathbb C} : |C|=1\}$ of the complex plane. It then follows from the main result in \cite{rodriguez} or Theorem 3.2 of \cite{stanleycycleperm} that each zero of $i(\Delta_{(1,\bq)}; t)$ has real part $\displaystyle -\frac{1}{2}.$ Therefore, the conclusion follows from the same discussion we give for cross-polytopes in subsection \ref{subsec: known families}.
\end{proof}

The proof of Theorem~\ref{thm: ehrhart positive but not ehrhart unimodal q} uses the same technique used to prove Ehrhart positivity of the $d$-dimensional cross-polytope $\Diamond_d$, discussed in Section~\ref{sec: polytopes that are ehrhart positive and ehrhart unimodal}.  
This special technique for proving Ehrhart positivity is the focus of an open problem posed in \cite{BR07}.  
In particular, Theorem~\ref{thm: ehrhart positive but not ehrhart unimodal q} provides one answer to \cite[Open Problem 2.43]{BR07}.

\subsection{Low dimensions}
\label{subsec: low dimensions}
As stated in section~\ref{sec: introduction}, the goal of this note is to assess the relationship between Ehrhart unimodality and Ehrhart positivity in each dimension greater than two.  
Since Theorems~\ref{thm: paynes theorem} and \ref{thm: ehrhart positive but not ehrhart unimodal q} only covers dimensions greater than five, it remains to identify examples of lattice polytopes that are Ehrhart positive but not Ehrhart unimodal in dimensions $3,4,$ and $5$. 
For these three dimensions, we then consider the \emph{Reeve's tetrahedron}, a well-known $3$-dimensional lattice simplex defined as follows:
Given a positive integer $h\geq1$ define the Reeve's tetrahedron
$$
\RR_h :=\conv((0,0,0),(1,0,0),(0,1,0),(1,1,h))\subset\R^3.
$$
It is well-known \cite[Example 3.22]{BR07} that the Reeve's tetrahedron $R_h$ has Ehrhart polynomial
$$
i(\RR_h;t) = \frac{h}{6}t^3+t^2+\left(\frac{12-h}{6}\right)t+1,
$$
and $h^\ast$-polynomial 
$$
h^\ast(\RR_h;z) = 1+(h-1)z^2.  
$$
In particular, for $2\leq h\leq11$ the Reeve's tetrahedron $\mathcal{R}_h$ is Ehrhart positive but not Ehrhart unimodal.  
These examples settle the question in dimension $3$.  
We can then lift this example into dimensions $4$ and $5$ by way of \emph{lattice pyramids} over $\mathcal{R}_h$.  
If $P\subset\R^n$ is a lattice polytope, then the \emph{lattice pyramid} over $P$ is the polytope  
\[
\Pyr(P) :=\conv(P\times\{0\},\be_{n+1})\subset\R^{n+1}.
\]
A well-known fact in Ehrhart theory is that $h^\ast(P;z) = h^\ast(\Pyr(P);z)$ \cite{B06}.  
We let $\Pyr^k(P)$ denote the $k$-fold pyramid over the lattice polytope $P$.
Using the software {\tt Polymake} \cite{Polymake} one can quickly compute that the four and five dimensional lattice pyramids $\Pyr(\mathcal{R}_6)$ and $\Pyr^2(\mathcal{R}_6)$ are both Ehrhart positive.  
This provides the remaining two examples needed to complete our objective in this section.
\begin{remark}
\label{rmk: summarizing}
The lattice polytopes $\mathcal{R}_6$, $\Pyr(\mathcal{R}_6)$, $\Pyr^2(\mathcal{R}_6)$, and those identified in Theorems~\ref{thm: paynes theorem} and \ref{thm: ehrhart positive but not ehrhart unimodal q} collectively demonstrate that in each dimension greater than two there exist lattice polytopes that are Ehrhart positive but not Ehrhart unimodal.  
\end{remark}
While our low-dimensional examples may suggest that the lattice pyramid operation preserves Ehrhart postivity (in addition to Ehrhart unimodality), we will see in the coming two sections that in fact quite the opposite is true.  


\section{Polytopes that are not Ehrhart Positive but are Ehrhart Unimodal}
\label{sec: polytopes that are not ehrhart positive but are ehrhart unimodal}
In this section, we demonstrate that there exist families of lattice polytopes that are Ehrhart positive but not Ehrhart unimodal in each dimension greater than two.  
We begin by showing, in subsection~\ref{subsec: smooth polytopes}, that there exists a smooth polytope in each dimension greater than six that is Ehrhart unimodal but not Ehrhart positive. 
Then, in subsection~\ref{subsec: s-lecture hall simplices}, we study Ehrhart positivity for the $s$-lecture hall simplices, which are a well-studied family of Ehrhart unimodal lattice polytopes \cite{Sav16}. 
We observe that in every dimension greater than two there are infinitely many $s$-lecture hall simplices that are not Ehrhart positive.  

\subsection{Smooth polytopes}
\label{subsec: smooth polytopes}
A $d$-dimensional lattice polytope $P$ is \emph{smooth} if every vertex of $P$ is contained in precisely $d$ edges and the primitive edge directions form a lattice basis for $\Z^d$.  
In \cite{CFNP17} the authors used the concept of \emph{chiseling} smooth lattice polytopes to obtain smooth polytopes with negative Ehrhart coefficients.  
Let $P$ be a smooth lattice polytope of dimension $d$ with vertex set $\Vert(P)$.  
Suppose that $\bv$ is a vertex of $P$ with primitive edge directions $\bu_1,\ldots,\bu_d$, and suppose also that there is an integer $b\in\Z_{>0}$ such that for all $i \in[n]$ the lattice point $\bv+b\bu_i$ is in $P$ but is not a vertex of $P$.  
The \emph{chiseling off of the vertex $\bv$ at distance $b$} from $P$ is the polytope
$$
P^\prime = \conv\left( (\Vert(P) \setminus \{\bv\})\cup\{\bv+b\bu_1,\ldots,\bv+b\bu_d\}\right).
$$
For each edge $e$ of $P$ define the \emph{lattice edge length} of $e$ to be $\ell(e) := |e\cap\Z^d|+1$, and let $\ell_P :=\min(\ell(e) : e \mbox{ is an edge of $P$})$.  
For any integer $b<\frac{\ell_P}{2}$, we can define the \emph{full chiseling} of $P$ to be the smooth polytope $\ch(P,b)$ produced by chiseling every vertex of $P$ at distance $b$.  
Using the theory of \emph{half-open decompositions} of lattice polytopes and the results of \cite{CFNP17}, we can find a smooth lattice polytope in each dimension greater than six that is Ehrhart unimodal but not Ehrhart positive.

For a $d$-dimensional lattice polytope $P$, let $P = P_1\cup\cdots\cup P_m$ be a \emph{decomposition} of $P$ into lattice polytopes $P_1,\ldots, P_m$; i.e., every $P_i$ is a $d$-dimensional lattice polytope such that $P_i$ does not intersect the relative interior of $P_j$ for any $j\neq i$.  
We say that a point $\bomega\in\R^d$ is in \emph{general position} with respect to the decomposition $P_1\cup\ldots\cup P_m$ if $\bomega$ does not lie in any facet-defining hyperplane for any $P_i$, $i \in[m]$.  
We further say that $\bomega$ is \emph{beyond} a facet $F$ of $P_i$ if the facet-defining hyperplane for $F$ separates $\bomega$ from the relative interior of $P_i$.  
Let $\Bey(P_i,\bomega)$ denote the collection of facets $F$ of $P_i$ for which $\bomega$ is beyond $F$.  
Then the \emph{half-open polytope} associated to $P_i$ and $\bomega$ is
$$
\mathbb{H}_\bomega P_i :=P_i\backslash\bigcup_{F\in\Bey(P_i,\bomega)}F.  
$$
\begin{lemma} 
\cite[Theorem 3]{KV08}
\label{lem: half-open decompositions}
Let $P$ be a lattice polytope and let $P = P_1\cup\cdots\cup P_m$ be a decomposition of $P$ into lattice polytopes.  
If $\bomega\in P$ is in general position with respect to $P_1\cup\cdots\cup P_m$ then $P$ is the disjoint union
$$
P = \mathbb{H}_\bomega P_1\sqcup\cdots\sqcup\mathbb{H}_\bomega P_m,
$$
and 
$$
h^\ast(P;z) = h^\ast( \mathbb{H}_\bomega P_1;z)+\cdots+h^\ast(\mathbb{H}_\bomega P_m;z). 
$$ 
\end{lemma}
The results of \cite{CFNP17} and Lemma~\ref{lem: half-open decompositions} allow us to identify the desired smooth lattice polytopes.  
In fact, the resulting examples have a very classical combinatorial flavor.
\begin{figure}
\centering
\includegraphics[width = 0.4\textwidth]{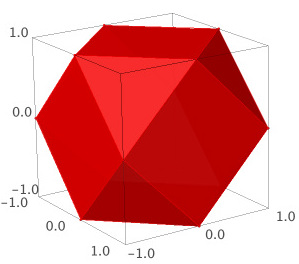}
\vspace{-0.2cm}
\caption{The chiseling of the cube $[-1,1]^3$ used in Theorem~\ref{thm: ehrhart unimodal but not ehrhart positive smooth polytopes}.  The same chiseling applied to $[-1,1]^d$ for $d\geq 7$ will fail to be Ehrhart positive but will still be Ehrhart unimodal.}
\label{fig: chiselling cube} 
\end{figure}
\begin{theorem}
\label{thm: ehrhart unimodal but not ehrhart positive smooth polytopes}
For $d\geq 7$, the chiseling $\ch([-1,1]^d,1)$ of the $d$-dimensional cube $[-1,1]^d$ is Ehrhart unimodal but not Ehrhart positive. 
\end{theorem}
\begin{proof}
It follows from the proof of \cite[Proposition 1.3]{CFNP17} that the linear term of $i(\ch([-1,1]^d,1);t)$ is 
$$
2d-\frac{2^d}{d},
$$
which is seen to be negative for $d\geq 7$.  
Therefore, we need only observe that the $h^\ast$-polynomial of $\ch([-1,1]^d,1)$ is unimodal.  
To see this, we first recall that the $h^\ast$-polynomial of the cube $[-1,1]^d$ is the \emph{Type B} Eulerian polynomial, which is well-known to be real-rooted and unimodal \cite{B94}.  
Recall from subsection~\ref{subsec: preliminaries} that the linear coefficient of the $h^\ast$-polynomial of a $d$-dimensional lattice polytope $P$ is always $|P\cap \Z|-(d+1)$.  
Thus, we know that
$$
[z].h^\ast([-1,1]^d;z)  = 3^d-(d+1),
$$
(Here, $[z^k].f(z)$ denotes the coefficient of $z^k$ in the polynomial $f(z).$)
Now, for each vertex $\bv$ of $[-1,1]^d$ with primitive edge directions $\bu_1,\ldots,\bu_d$, define the unimodular $d$-simplex $S_\bv :=\conv(\bv,\bv+\bu_1,\ldots,\bv+\bu_d)$.  
It follows that $[-1,1]^d$ admits the decomposition into lattice polytopes
$$
[-1,1]^d = \ch([-1,1]^d,1)\cup\bigcup_{\bv\in\Vert([-1,1]^d)}S_\bv,
$$
and the origin $\0\in\R^d$ is in general position with respect to this decomposition for all $d\geq2$.  
Thus, when $d\geq 2$, $\0$ is beyond no facet of $\ch([-1,1]^d,1)$, and for all $\bv\in\Vert([-1,1]^d)$, $\0$ is only beyond the facet $F_\bv$ of $S_\bv$ that does not contain $\bv$.  
In particular, $h^\ast(\mathbb{H}_\0\ch([-1,1]^d,1);z) = h^\ast(\ch([-1,1]^d,1);z)$, and for all $\bv$ in $\Vert([-1,1]^d)$ we have that $h^\ast(\mathbb{H}_\0S_\bv; z) = z$; i.e., the $h^\ast$-polynomial of the standard $d$-simplex with precisely one facet removed.  
It then follows from Lemma~\ref{lem: half-open decompositions} that
$$
h^\ast(\ch([-1,1]^d,1);z) = h^\ast([-1,1];z)-2^dz.
$$
Since $3^d-2^d-(d+1)>0$ for all $d\geq 7$ and $h^\ast([-1,1];z)$ is unimodal, it follows that $h^\ast(\ch([-1,1]^d,1);z)$ is also unimodal. 
Thus, for every dimension $d\geq 7$, $\ch([-1,1]^d,1)$ is Ehrhart unimodal but not Ehrhart positive.
\end{proof}

\subsection{$s$-Lecture hall simplices}
\label{subsec: s-lecture hall simplices}
A well-studied family of examples of Ehrhart unimodal lattice polytopes are the \emph{$\bs$-lecture hall simplices} \cite{SS12}.  
Let $\bs := (s_k)_{k=1}^{d}$ be a sequence of positive integers. 
The $d$-dimensional \emph{$\bs$-lecture hall simplex} is 
\[
P_d^\bs := \left\{\bx\in\R^d : 0\leq\frac{x_1}{s_1}\leq\frac{x_2}{s_2}\leq\cdots\leq\frac{x_d}{s_d}\leq 1\right\}\subset\R^d.
\]
The $h^\ast$-polynomial of $P_d^s$ is called the \emph{$\bs$-Eulerian polynomial}, and it enumerates 
the \emph{$\bs$-inversion sequences} 
\[
	\mathcal{J}_d^\bs :=\left\{\m\in\Z^d : 0\leq m_i<s_i\right\}
\]
by their number of \emph{$\bs$-ascents}; that is, the value
\[
	\asc_\bs(\m) := \left|\left\{i\in\{0,1,\ldots,d-1\} : \frac{m_i}{s_i}<\frac{m_{i+1}}{s_{i+1}}\right\}\right|,
\]
for $\m\in\mathcal{J}_d^\bs$, with the convention that $m_0 := 0$ and $s_0:=1$ \cite{SS12}.  
In other words,
\begin{equation}
\label{equ:hstar-slh}
h^\ast(P_d^\bs; z) = \sum_{\m\in \mathcal{J}_d^\bs} z^{\asc_\bs(\m)}.
\end{equation}
In the case that $\bs = (1,2,\ldots,d)$, the $\bs$-Eulerian polynomial is the classic \emph{Eulerian polynomial}, which enumerates the permutations of $[n]$ by the descent statistic.  
It was shown in \cite{SV15} that for all choices of $\bs$ and $d$, the $h^\ast$-polynomial of $P_d^\bs$ is real-rooted and therefore unimodal.  
The $\bs$-lecture hall simplices are a combinatorially rich family of lattice polytopes (see for example \cite{Sav16}), so it is natural to ask whether or not they are of interest from the perspective of Ehrhart positivity as well.  
In fact, as we see with the following theorem, there exist infinitely many $\bs$-lecture hall simplices, even in low dimensions, that are not Ehrhart positive.  
In the following, for positive integers $a, b, k_2$ and nonnegative integers $k_1, k_3$, we write $(1^{k_1}, a, 1^{k_2}, b, 1^{k_3})$ for the vector $(\underbrace{1, \dots, 1}_{k_1}, a, \underbrace{1, \dots, 1}_{k_2}, b, \underbrace{1, \dots, 1}_{k_3}).$

\begin{theorem} 
\label{thm: thin lecture hall simplices}
For every $d\geq3$, there exist $\bs$-lecture hall simplices of the form $P_d^{(1^{k_1}, a, 1^{k_2}, b, 1^{k_3})}$ that are not Ehrhart positive.  
\end{theorem}

We will use the following lemma to prove the above theorem, which allows us to write $h^\ast(P_d^{(1^{k_1}, a, 1^{k_2}, b, 1^{k_3})};z)$ explicitly in terms of the parameters $a$ and $b$.
\begin{lemma}
\label{lem:hstar-a-1-b}
For any positive integers $a, b , k_2$, nonnegative integers $k_1, k_3$ and integer $d \ge 3,$ we have
\begin{equation}
	h^\ast(P_d^{(1^{k_1}, a, 1^{k_2}, b, 1^{k_3})}; z) = (1 + (a-1) z) (1 + (b-1) z).
	\label{equ:hstar-a-1-b}
\end{equation}
\end{lemma}
\begin{proof}
	It follows from equation~(\ref{equ:hstar-slh}) that 
	\[
	h^\ast(P_d^{(1^{k_1}, a, 1^{k_2}, b, 1^{k_3})}; z)  = h^\ast(P_d^{(a, 1^{k_2}, b)}; z).  
	\]
	Thus it suffices to prove the statement for $\bs=(a,1^{d-2},b).$ 
	Suppose $\m\in\mathcal{J}_d^\bs$. 
	Then $m_1 \in \{0,1,\dots,a-1\}, m_d =\{0, 1, \dots, b-1\}$, and $m_i =0$ for $2 \le i \le d-1.$ 
	Notice that only $0$ and $d-1$ can be $\bs$-ascents. 
	Furthermore, $0$ is an $\bs$-ascent in $m$ if and only if $m_1 \neq 0,$ and $d-1$ is an $\bs$-ascent in $m$ if and only if $m_d \neq 0.$ 
	Therefore,
	\[ h^\ast(P_d^{(a, 1^{d-2},b)};z) = \sum_{\m\in \mathcal{J}_d^\bs} z^{\asc_\bs(\m)} = \sum_{m_1=0}^{a-1} z^{f(m_1)} \sum_{m_d=0}^{b-1} z^{f(m_d)},\]
	where $f(x) = 0$ if $x = 0$, and $f(x)=1$ if $x > 0.$ 
	We then see that the right-hand-side of the above equation becomes $(1 + (a-1) z)(1+(b-1)z).$
\end{proof}

The proof of Theorem~\ref{thm: thin lecture hall simplices} is then given as follows.  
\begin{proof}[Proof of Theorem~\ref{thm: thin lecture hall simplices}]
We will show that for any integer $a > d,$ the $\bs$-lecture hall simplex $P_d^{(1^{k_1}, a, 1^{k_2}, b, 1^{k_3})}$ is Ehrhart positive for sufficiently large integer $b.$
It follows from Lemma \ref{lem:hstar-a-1-b} and Formula \eqref{equ:h2e} that for any integers $a, b>1$
\[
	i(P_d^{(1^{k_1}, a, 1^{k_2}, b+1, 1^{k_3})};t) - i(P_d^{(1^{k_1}, a, 1^{k_2}, b, 1^{k_3})};t) = \binom{t+d-1}{d} + (a-1) \binom{t+d-2}{d},
\]
which is independent of the parameter $b$.  
Therefore, it suffices to show that the linear term of the above expression is always negative. However,
\[ [t].\left(  \binom{t+d-1}{d} + (a-1) \binom{t+d-2}{d}\right) = \frac{d-a}{d(d-1)},\]
which is negative if $a>d.$ This completes the proof.
\end{proof}


Since all $\bs$-lecture hall simplices in Theorem~\ref{thm: thin lecture hall simplices} have lattice width one, we can think of them as \emph{thin} $\bs$-lecture hall simplices.
In the special case of Theorem~\ref{thm: thin lecture hall simplices} when $d=3$, we can further use Lemma~\ref{lem:hstar-a-1-b} to explicitly compute the Ehrhart polynomial $i(P_3^{(a, 1, b)};t)$, from which we can identify a spectrahedral cone containing all lattice points $(a,b)\in\Z^2_{>0}$ such that $P_3^{(a, 1, b)}$ is not Ehrhart positive.  
\begin{corollary}
\label{thm: lecture hall simplices}
Suppose that $\bs = (a, 1, b)$.
Then the $3$-dimensional $\bs$-lecture hall simplex $P_3^\bs$ has Ehrhart polynomial
$$
i(P_3^\bs;t) = \left(\frac{ab}{6}\right)t^3 +\left(\frac{a+b}{2}\right)t^2+\left(\frac{6+3(a+b)-ab}{6}\right)t+ 1.
$$
In particular, all such $\bs$-lecture hall simplices that are not Ehrhart positive are given by the lattice points $(a,b)\in\Z^2_{>0}$ satisfying $6+3(a+b)-ab < 0$.  
\end{corollary}
\begin{figure}
\centering
{\subfigure{\includegraphics[width = 0.25\textwidth]{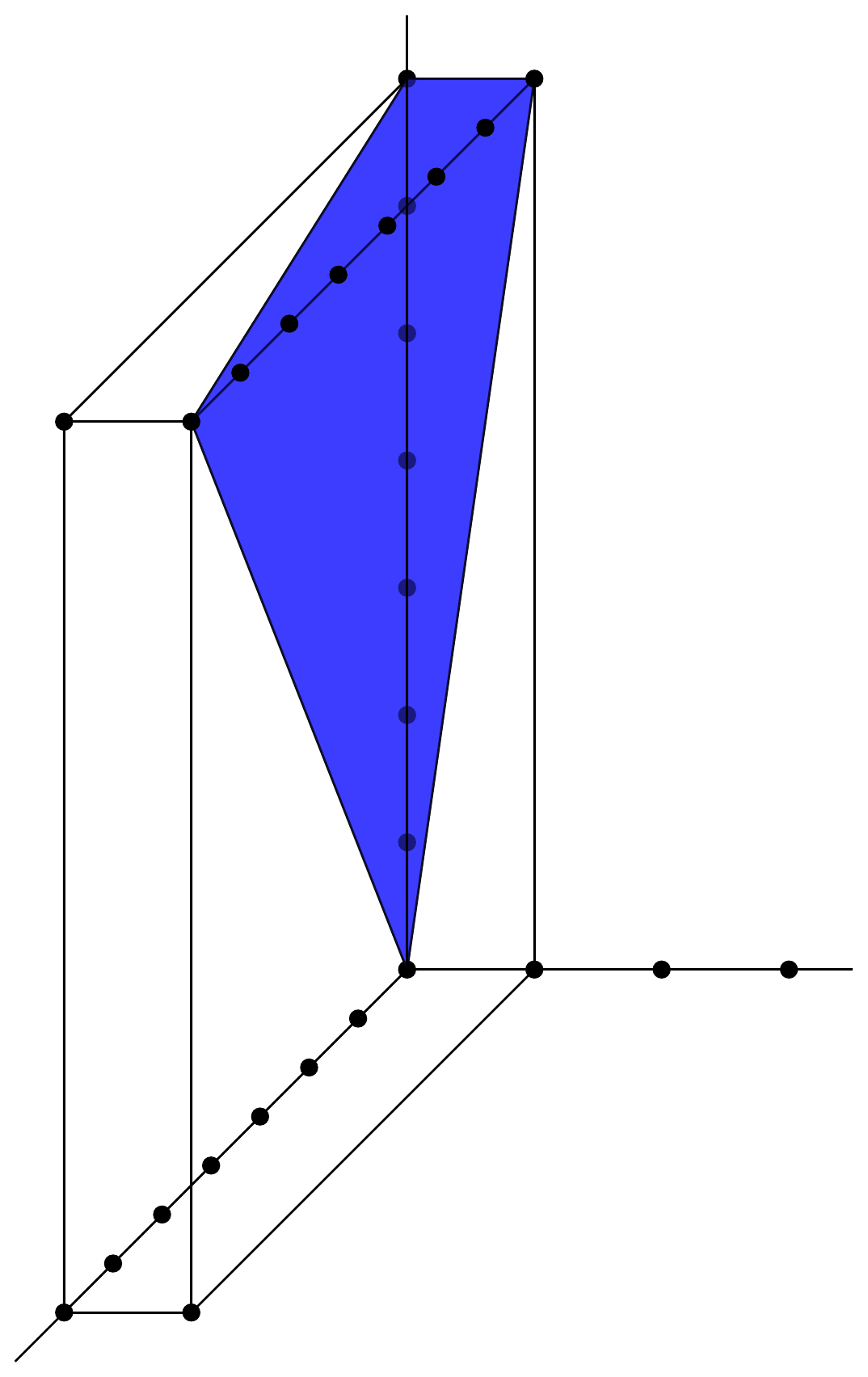}}\qquad\qquad
\subfigure{\includegraphics[width = 0.4\textwidth]{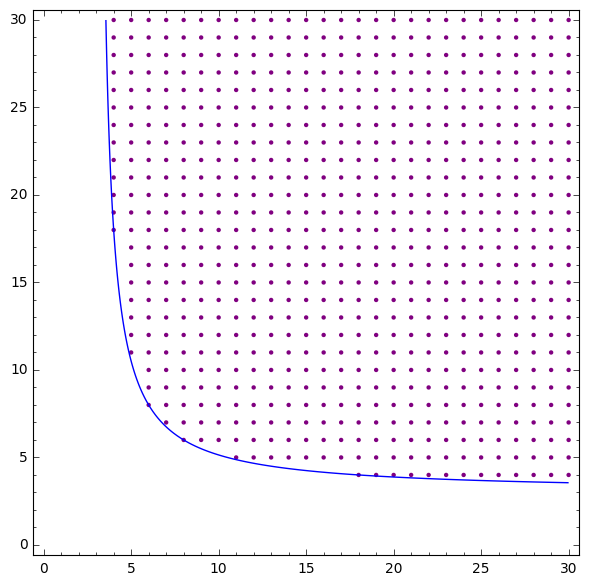}}}
\vspace{-0.2cm}
\caption{On the left we see the thin lecture hall simplex $P_3^{(7,1,7)}$, and on the right we see the cone of lattice points $(s_1,s_3)$ for which $P_3^{(s_1,1,s_3)}$ is not Ehrhart positive.}
\label{fig: lecture hall simplices} 
\end{figure}
Figure~\ref{fig: lecture hall simplices} depicts the spectrahedral cone in $\R^2$ defined by the linear term of $i(P_3^\bs;t)$ in Corollary~\ref{thm: lecture hall simplices} that captures the collection of lattice points $(a,b)$ yielding $P_3^\bs$ that are not Ehrhart positive.  
From this picture we can see that the vast majority of $\bs$-lecture hall $3$-simplices that are thin in the second coordinate are not Ehrhart positive. 
We also see from this corollary that the Ehrhart polynomials of these $\bs$-lecture $3$-simplices are similar to that of the Reeve's tetrahedra, a family of lattice $3$-simplices that were introduced in subsection~\ref{subsec: low dimensions}.
We end this section with a geometric remark that further connects the $\bs$-lecture hall $3$-simplices $P_3^{(a,1,b)}$ to the Reeve's tetrahedron.  
\begin{remark}
[Pyramids and $\bs$-lecture hall simplices]
\label{rmk: pyramids and lecture hall simplices}
Recall from Subsection~\ref{subsec: low dimensions} that if $\Pyr^k(P)$ is the $k$-fold lattice pyramid over a lattice polytope $P$ then $h^\ast(P;z) = h^\ast(\Pyr(P);z)$ \cite{B06}.  
Given an $\bs$-lecture hall simplex $P_d^{(s_1,\ldots,s_d)}$, notice that the $\bs$-lecture hall simplices $P_{d+1}^{(1,s_1,\ldots,s_d)}$ and $P_{d+1}^{(s_1,\ldots,s_d,1)}$ are lattice pyramids over $P_d^{(s_1,\ldots,s_d)}$.  
This gives an alternative method by which to observe that 
\[
h^\ast(P_d^{(1^{k_1}, a, 1^{k_2}, b, 1^{k_3})}; z)  = h^\ast(P_d^{(a, 1^{k_2}, b)}; z),
\]
a fact that we used in the proof of Theorem~\ref{thm: thin lecture hall simplices}.  
On the other hand, $P_d^{(a,1^{d-2},b)}$ is not a $(d-2)$-fold lattice pyramid over $P_2^{(a,b)}$, as can already be seen for $d = 3$ in the left-hand-side of Figure~\ref{fig: lecture hall simplices}.  
These examples demonstrate how lattice pyramids can be used to recover non-Ehrhart positive lattice polytopes in high dimensions with a chosen Ehrhart $h^\ast$-polynomial.  
Analogous to these $s$-lecture hall simplices, in the coming section, we will use the pyramid construction in relation to the Reeve's tetrahedra to derive our final collection of examples.  
\end{remark}




%

\section{Polytopes that are neither Ehrhart Positive nor Ehrhart Unimodal}
\label{sec: polytopes that are neither ehrhart positive nor ehrhart unimodal}


In this section we present a family of lattice polytopes containing polytopes in each dimension greater than two that are neither Ehrhart positive nor Ehrhart unimodal.  
Analogous to the previous sections, these polytopes also have a nice geometric construction that relies on fundamental examples and tools used frequently in polyhedral geometry and Ehrhart theory.  
Recall that a lattice simplex is called \emph{empty} if it contains no lattice points apart from its vertices.  
In the following, we show that there exist infinitely many empty lattice simplices in each dimension greater than two that are neither Ehrhart positive nor Ehrhart unimodal.  

\subsection{Empty simplices}
\label{subsec: empty simplices}
In dimension three there exists a well-known family of empty lattice simplices that are neither Ehrhart positive nor Ehrhart unimodal.  
This family is collectively known as the \emph{Reeve's tetrahedra}, which we introduced in subsection~\ref{subsec: low dimensions}.    
Recall from subsection~\ref{subsec: low dimensions} that the Reeve's tetrahedron $\mathcal{R}_h$ has $h^\ast$-polynomial 
$$
h^\ast(\RR_h;z) = 1+(h-1)z^2.  
$$
Thus, the Reeve's tetrahedron $\RR_h$ exhibits the special shape of the $h^\ast$-polynomial of empty simplices; namely, $h_1^\ast = 0$ for any empty lattice simplex.  
Moreover, any nonunimodular empty lattice $d$-simplex will not be Ehrhart unimodal.  
In addition, recalling that 
$$
i(\RR_h;t) = \frac{h}{6}t^3+t^2+\left(\frac{12-h}{6}\right)t+1,
$$
we see that for any $h\geq13$ the Reeve's tetrahedron $\RR_h$ will also not be Ehrhart positive.   
The following theorem shows that both phenomena can be lifted into higher dimensions using the techniques we applied to $\bs$-lecture hall simplices in subsection~\ref{subsec: s-lecture hall simplices}, and the pyramid construction defined in subsection~\ref{subsec: low dimensions}.    
 \begin{theorem}
 \label{thm: pyramids over reeves tetrahedron}
 For $d\geq 3$ let 
 \[
 H := \left\lceil\frac{1}{(d-2)!}{d+1 \brack 2}\right\rceil+1,
 \]
 where ${n \brack k}$ denotes the unsigned Stirling number of the first kind.  
 For all $h\geq H$ the empty lattice $d$-simplex
 \[
 \Pyr^{d-3}(\RR_h)
 \]
 is neither Ehrhart positive nor Ehrhart unimodal.  
 \end{theorem}
 
 \begin{proof}
 The proof of this theorem is another application of the techniques applied in subsection~\ref{subsec: s-lecture hall simplices}.  
 First, recall that for all $d-3\geq0$ the $(d-3)$-fold pyramid over $\RR_h$ satisfies
 $$
 h^\ast\left(\Pyr^{d-3}(\RR_h);z\right) = h^\ast\left(\RR_h;z\right) = 1+(h-1)z^2,
 $$
 and therefore $\Pyr^{d-3}(\RR_h)$ is not Ehrhart unimodal.  
 Again by \eqref{equ:h2e}, we have
 \[  i\left(\Pyr^{d-3}(\RR_h);t\right) = \binom{t+d}{d} + (h-1)\binom{t+d-2}{d}.\]
Hence,
\[ i\left(\Pyr^{d-3}(\RR_{h+1});t\right) -  i\left(\Pyr^{d-3}(\RR_h);t\right) = \binom{t+d-2}{d}.\]
Similar to the proof of  Theorem \ref{thm: thin lecture hall simplices}, we can show that the linear term of the above expression is always negative. 
In particular, we notice that
\[ 
[t]. \binom{t+d-2}{d} = -\frac{1}{d(d-1)} < 0.  
\]
It then follows that for some sufficiently large $\mathfrak{h}\in\Z_{>0}$, the lattice pyramid $\Pyr^{d-3}(\RR_h)$ is not Ehrhart positive if and only if $h\geq \mathfrak{h}$.  
Thus, it remains to prove that the value $\mathfrak{h}$ is indeed $H$, as defined above.  
To check this, notice that
\[
[t].i\left(\Pyr^{d-3}(\RR_1);t\right) = [t].{t+d\choose d} = \sum_{k=1}^d\frac{1}{k}. 
\]
So by the above argument we have that for all $h>1$,
\[
[t].i\left(\Pyr^{d-3}(\RR_h);t\right) = \sum_{k=1}^d\frac{1}{k} - \frac{h-1}{d(d-1)}.
\]
Thus, the linear term of $i\left(\Pyr^{d-3}(\RR_h);t\right)$ is negative whenever 
\[
\sum_{k=1}^d\frac{1}{k} - \frac{h-1}{d(d-1)}<0.
\]
Since, 
\[
\sum_{k=1}^d\frac{1}{k} = \frac{1}{d!}{d+1\brack2},
\]
this completes the proof.
 \end{proof}

\section{Final Remarks}
\label{sec: final remarks}
In this note, we examined the relationship between the properties of Ehrhart unimodality and Ehrhart positivity of lattice polytopes within each dimension greater than (or equal to) two. 
We focused on well-studied families of polytopes that were previously investigated with respect to one property but not the other.  
These families of polytopes included simplices of weighted projective spaces, smooth polytopes, $s$-lecture hall simplices, and empty lattice simplices arising as $k$-fold pyramids over the well-known Reeve's tetrahedron.  
Through this analysis, we showed that in each dimension greater than two there is no relationship between the properties of Ehrhart unimodality and Ehrhart postivity.  
That is, in each such dimension there exists a lattice polytope that is (1) both Ehrhart positive and Ehrhart unimodal, (2) Ehrhart positive but not Ehrhart unimodal, (3) Ehrhart unimodal but not Ehrhart positive, and (4) neither Ehrhart positive nor Ehrhart unimodal. 
These results provide new examples in regards to both Ehrhart unimodality and Ehrhart positivity for well-studied families of lattice polytopes, and at the same time make explicit the relationship (or lack thereof) between these two properties with respect to dimension.  

On the other hand, the results in this note do not exclude the possibility that there exist special families of polytopes for which there is some implication between Ehrhart unimodality and Ehrhart positivity.  
Such examples would be of general interest, since they would constitute a setting in which techniques for proving one property are utilizable in the analysis of the other.  
In this direction, one useful tool pointed out by the various examples in this paper that could pertain to such case analyses is the lattice pyramid operation.  
Suppose we are interested in analyzing the Ehrhart polynomials of a collection $\Omega$ of lattice polytopes.  
The examples presented here suggest that if $\Omega$ (or a subset thereof) is closed under lattice pyramids, then this operation can be used to identify members of $\Omega$ exhibiting both Ehrhart positivity and non-Ehrhart positivity.  
This purports the lattice pyramid operation not only as a useful tool in analyzing the shape of $h^\ast$-polynomials, but also in assessing the likely validity of conjectures on Ehrhart positivity for special families of polytopes.  


\smallskip

\noindent
{\bf Acknowledgements}.
Fu Liu was partially supported by a grant from the Simons Foundation \#426756. She was also supported by the National Science Foundation under Grant No. DMS-1440140 while she was in residence at the Mathematical Sciences Research Institute in Berkeley, California, during the Fall 2017 semester.
Liam Solus was supported by an NSF Mathematical Sciences Postdoctoral Research Fellowship (DMS - 1606407). 
The authors would like to thank Alexander Postnikov for posing the motivating question of this note.  
This research began at the 2017 MSRI Introductory Workshop to the Semester on Geometric and Topological Combinatorics and was continued at the 2017 Mini-Workshop on Lattice Polytopes at Oberwolfach.  
The authors are grateful to both institutions as well as the organizers of each event.

%
%

\begin{thebibliography}{99}

\bibitem{B06}
	V. V. Batyrev. 
	\emph{Lattice polytopes with a given $h^\ast$-polynomial.}
	Algebraic and geometric combinatorics 423 (2006): 1-10.

\bibitem{BJM16}
	M. Beck, K. Jochemko, and E. McCullough. 
	\emph{$h^\ast$-polynomials of zonotopes.}
	arXiv preprint arXiv:1609.08596 (2016).
	
	
\bibitem{BR07}
	M. Beck and S. Robins. 
	\emph{Computing the continuous discretely.} 
	Springer Science + Business Media, LLC, 2007.
%
%

\bibitem{B16}
	B. Braun.
	\emph{Unimodality problems in Ehrhart theory.}
	Recent trends in combinatorics. Springer International Publishing, 2016. 687-711.
	
\bibitem{BD16}
	B. Braun and R. Davis. 
	\emph{Ehrhart series, unimodality, and integrally closed reflexive polytopes.}
	Annals of Combinatorics 20.4 (2016): 705-717.
%
\bibitem{BDS16} 
	B. Braun, R. Davis, and L. Solus.
	\emph{Detecting the integer decomposition property and Ehrhart unimodality in reflexive simplices.}
	Preprint available at \url{https://arxiv.org/abs/1608.01614} (2016).
%
\bibitem{B15}
	P. Br\"and\'en. 
	\emph{Unimodality, log-concavity, real-rootedness and beyond.}
	Handbook of Enumerative Combinatorics (2015): 437-483.

\bibitem{B89}
	F. Brenti.
	\emph{Unimodal log-concave and Pólya frequency sequences in combinatorics.}
	No. 413. American Mathematical Soc., 1989.
	
\bibitem{B94}
	\bysame, 
	\emph{$q$-Eulerian polynomials arising from Coxeter groups.} 
	European J. Comb., 15:417--441, September 1994.

\bibitem{BValpha}
F.~Castillo and F.~Liu, \emph{{B}erline-{V}ergne valuation and generalized
  permutohedra}, Discrete Comput. Geom., accepted.


\bibitem{ehrhartpos-gp-fpsac}
\bysame, \emph{Ehrhart positivity for generalized permutohedra}, Discrete Math
  Theor. Comput. Sci. proc. \textbf{FPSAC '15} (2015), 865--876.

\bibitem{CFNP17}
	F. Castillo, F. Liu, B. Nill, and A. Paffenholz. 
	\emph{Smooth polytopes with negative Ehrhart coefficients.}
	arXiv preprint available at \url{https://arxiv.org/pdf/1704.05532.pdf} (2017).
%
%
%

\bibitem{DHK09}
	J. A. De Loera, D. C. Haws, and M. K\"oppe. 
	\emph{Ehrhart polynomials of matroid polytopes and polymatroids.}
	Discrete \& computational geometry 42.4 (2009): 670-702.
	
\bibitem{E62}
	E. Ehrhart.  
	\emph{Sur les polyh\`{e}dres rationnels homoth\`{e}tiques \`{a} $n$ dimensions}.  
	C. R. Acad Sci. Paris, 254:616-618, 1962.
	
\bibitem{Polymake}
	E. Gawrilow and M. Joswig. 
	\emph{Polymake: a framework for analyzing convex polytopes.}
	Polytopes--combinatorics and computation. Birkh\"auser, Basel, 2000.
	
\bibitem{H92}
	T. Hibi.
	\emph{Algebraic combinatorics on convex polytopes.}
	Carslaw publications, 1992.

\bibitem{HHTY15}
	T. Hibi, A. Higashitani, A. Tsuchiya, and K. Yoshida. 
	\emph{Ehrhart polynomials with negative coefficients.} 
	arXiv preprint available at \url{https://arxiv.org/pdf/1506.00467.pdf} (2015).
	
\bibitem{KMR17}
	K. Knauer, L. Martínez-Sandoval, and J. L. Ramírez Alfonsín. 
	\emph{On lattice path matroid polytopes: integer points and Ehrhart polynomial.}
	arXiv preprint available at \url{https://arxiv.org/pdf/1701.05529.pdf} (2017).
	
\bibitem{KV08}
	M. K\"oppe and S. Verdoolaege.  
	\emph{Computing parametric rational generating functions with a primal Barvinok algorithm.}
	Electron. J. Comb in., 15 (2008).
%
%
%

\bibitem{L05}
	F. Liu
	\emph{Ehrhart polynomials of cyclic polytopes.}
	Journal of Combinatorial Theory, Series A 111.1 (2005): 111-127.
	
\bibitem{L17}
	\bysame,
	\emph{On positivity of Ehrhart polynomials.}
	 arXiv preprint arXiv:1711.09962 (2017).
	
\bibitem{Ober17}
	Mini-Workshop: \emph{Lattice Polytopes: Methods, Advances, Applications}.
	Abstracts from the mini-workshop held September 17?23, 2017. 
	Organized by T. Hibi, A. Higashitani, K. Jochemko, and B. Nill. Oberwolfach Reports. no. 44. 
	(2017).

	
\bibitem{P08}
	S. Payne.
	\emph{Ehrhart series and lattice triangulations.}
	Discrete \& Computational Geometry 40.3 (2008): 365-376.
%
%
%
%
%
%
%

\bibitem{rodriguez}
F.~R. Rodriguez-Villegas, \emph{On the zeros of certain polynomials}, 
Proc.  Amer. Math. Soc. \textbf{130} (2002), 2251--2254.

\bibitem{Sav16} 
	C. ~D. ~Savage. 
	\emph{The mathematics of lecture hall partitions}. 
	Journal of Combinatorial Theory, Series A, {\bf 144} (2016), 443--475.
	
\bibitem{SS12}
	C. ~D. ~Savage, M. ~J. ~Schuster,
	\emph{Ehrhart series of lecture hall polytopes and Eulerian polynomials for inversion sequences},
	Journal of Combinatorial Theory, Series A {\bf 119} (2012), 850--870.

\bibitem{SV15}
	C. ~D. ~Savage, M. ~Visontai, 
	\emph{The $s$-Eulerian polynomials have only real roots,} 
	Trans. Amer. Math. Soc., {\bf 367}(2), (2015), 1441--1466.

\bibitem{S17}
	L. Solus.
	\emph{Simplices for numeral systems.}
	To appear in Transactions of the American Mathematical Society (2017).  
	
\bibitem{S80}
	R. P. Stanley. 
	\emph{Decompositions of rational convex polytopes.}
	Annals of discrete mathematics 6 (1980): 333-342.

\bibitem{S97}
	\bysame,
	\emph{Enumerative Combinatorics.} 
	Vol. 1, vol. 49 of Cambridge Studies in Advanced Mathematics (1997).

\bibitem{S15}
	\bysame,
	\emph{On the positivity of Ehrhart polynomial coefficients.}
	\url{https://mathoverflow.net/questions/185723/positivity-of-ehrhart-polynomial-coefficients} (2015).

\bibitem{stanleycycleperm}
\bysame, \emph{Two enumerative results on cycles of permutations}, European J.
  Combin. \textbf{32} (2011), no.~6, 937--943. 

\bibitem{Z12} 
	G. M. Ziegler.
	\emph{Lectures on Polytopes.}
	Vol. 152. Springer Science \& Business Media (2012).
%

\end{thebibliography}
\end{document}